\documentclass[num-refs]{wiley-article}

\usepackage{color}
\usepackage{ulem}

\newcommand{\add}[1]{\textcolor{blue}{#1}}

\usepackage{amssymb}
\usepackage{dsfont}
\usepackage{hyperref}
\usepackage{lineno}
\usepackage{siunitx}

\usepackage{mathrsfs}

\newcommand{\N}{\mathbb{N}}
\newcommand{\R}{\mathbb{R}}
\newcommand{\PP}{\mathsf{P}} 
\newcommand{\QQ}{\mathsf{Q}} 
\newcommand{\EE}{\mathsf{E}} 
\newcommand{\Var}{\mathsf{Var}} 
\newcommand{\bb}[1]{\boldsymbol{#1}}
\newcommand{\OO}{\mathcal{O}}
\newcommand{\rd}{{\rm d}}
\newcommand{\ind}{\mathds{1}}

\allowdisplaybreaks

\papertype{Original Article}
\paperfield{Journal Section}

\title{Asymptotic comparison of negative multinomial \\and multivariate normal experiments}

\author[1\authfn{1}]{Christian Genest}
\author[2\authfn{1}]{Fr\'ed\'eric Ouimet}

\affil[1]{Department of Mathematics and Statistics, McGill University, Montr\'eal (Qu\'ebec) Canada H3A 0B9}
\affil[2]{Centre de recherches math\'ematiques, Universit\'e de Montr\'eal, Montr\'eal (Qu\'ebec) Canada H3T 1J4}

\corraddress{Christian Genest \\ Department of Mathematics and Statistics \\ McGill University \\
805, rue Sherbrooke ouest \\ Montr\'eal (Qu\'ebec) \\ Canada H3A 0B9}
\corremail{christian.genest@mcgill.ca}

\fundinginfo{Canada Research Chairs Program, Grant Number: 950-231937; Natural Sciences and Engineering Research Council of Canada, Grant Number: RGPIN-2016-04720; CRM-Simons Postdoctoral Fellowship Program.}

\runningauthor{Genest \& Ouimet}

\begin{document}

\maketitle

\begin{abstract}
This note presents a refined local approximation for the logarithm of the ratio between the negative multinomial probability mass function and a multivariate normal density, both having the same mean-covariance structure. This approximation, which is derived using Stirling's formula and a meticulous treatment of Taylor expansions, yields an upper bound on the Hellinger distance between the jittered negative multinomial distribution and the corresponding multivariate normal distribution. Upper bounds on the Le\,Cam distance between negative multinomial and multivariate normal experiments ensue.

\keywords{Asymptotic theory, comparison of experiments, Le\,Cam distance, local limit theorem, multivariate normal distribution, negative multinomial distribution.}
\end{abstract}

\section{Introduction\label{sec:1}}

Given an integer $d \in \N = \{1, 2, \ldots \}$, consider an experiment that generates $d + 1 \ge 2$ possible (mutually disjoint) outcomes occurring with positive probabilities $p_0, \ldots, p_d \in (0, 1)$, respectively. If sampling proceeds randomly until the number of outcomes of type $0$ reaches the predetermined value $r \in \mathbb{N}$, the random (column) vector $\bb{K} = (K_1, \ldots, K_d)$ recording the number of items of types $1$ through $d$ is then said to have a negative multinomial distribution with parameters $r \in \mathbb{N}$ and $\bb{p} = (p_1, \ldots, p_d)$, denoted $\bb{K} \sim \mathcal{N\hspace{-0.8mm}M} (r, \bb{p})$, where $\|\bb{p}\|_1 = p_1 + \cdots + p_d = 1 - p_0 < 1$.

Introduced independently by \citet{MR33996}, \citet{MR50836}, and \citet{MR53459}, this flexible urn model extends the classical multinomial distribution by allowing for overdispersion and zero inflation. It has found applications in ecology and genetics, but also in many other contexts involving categorical data. The case $d = 1$ corresponds to the classical negative binomial distribution.

An extension of the $\mathcal{N\hspace{-0.8mm}M} (r,\bb{p})$ distribution can be obtained, for any real-valued parameter $r \in (0, \infty)$, by defining its probability mass function, for any vector $\bb{k} = (k_1, \ldots, k_d) \in \N_0^d = \{0, 1, \ldots \}^d$, by
\begin{equation}
\label{eq:1}
P_{r, \bb{p}}(\bb{k}) = \frac{\Gamma(r + \|\bb{k}\|_1)}{\Gamma(r) \prod_{i=1}^d k_i!} \, (1 - \|\bb{p}\|_1)^r \prod_{i=1}^d p_i^{k_i} = \frac{\Gamma(r + \|\bb{k}\|_1)}{\Gamma(r) \prod_{i=1}^d k_i!} \, p_0^{r + \|\bb{k}\|_1} \prod_{i=1}^d \varrho_i^{k_i},
\end{equation}
where $\Gamma$ denotes Euler's gamma function and $\varrho_1 = p_1 / p_0, \ldots, \varrho_d = p_d / p_0$. The main characteristics of this model are summarized, e.g., in \cite{MR171341} and \cite[Chapter 36]{MR1429617}. In particular, the mean \add{vector} and covariance matrix of this distribution are respectively given by $\bb{\mu} = r \bb{\varrho}$ and $r \Sigma$, where $\bb{\varrho} = (\rho_1, \ldots, \rho_d)$ and $\Sigma = \text{diag} (\bb{\varrho}) + \bb{\varrho} \bb{\varrho}^{\top}$.

It has long been known, see, e.g., \cite{MR0251828, MR1219279}, that for sufficiently large values of $r$, the probability mass function of the negative multinomial distribution given in~\eqref{eq:1} can be approximated by a multivariate normal density with the same mean-covariance structure, denoted $\mathcal{N}_d(r \bb{\varrho}, r \Sigma)$.

The aim of this note is to quantify more precisely the distance between the two distributions as $r \to \infty$. General results on local asymptotic expansions of probabilities associated with sums of lattice random vectors already exist; see, e.g., \citep{MR0436272, doi:10.1137/1114060, doi:10.1007/BF00967926}. However, the error terms in these expansions are not explicit, which hampers their practicality. To address this limitation and enhance applicability, the specific form of the negative multinomial distribution will be leveraged. This will lead to explicit error terms and rates of convergence which refine existing results. The machinery involved is also much simpler and, therefore, should be transparent to users. Specifically, this paper's contribution is threefold.

In Proposition~\ref{prop:1}, a refined local approximation for the log-ratio of the negative multinomial probability mass function over the corresponding multivariate normal density function is proved, based on Stirling's formula and a careful handling of Taylor expansions. In Proposition~\ref{prop:2}, the new refined local approximation is combined with moment estimates to find an upper bound on the Hellinger distance between the jittered negative multinomial distribution and the corresponding multivariate normal distribution. In Proposition~\ref{prop:3}~and~Corollary~\ref{cor:1}, upper bounds on the Le\,Cam distance between negative multinomial and multivariate normal experiments are derived.

These results are stated in Section~\ref{sec:2} and comments regarding their significance are given in Section~\ref{sec:3}. All proofs are in Appendix~\ref{App:1} and moment estimates needed for the proof of Proposition~\ref{prop:2} are relegated to Appendix~\ref{App:2}. In what follows, the notation $u_r = \OO(v_r)$ means that there exists a universal constant $C\in (0,\infty)$ such that $\limsup_{r\to \infty} |u_r / v_r | \le C$. Any dependence of $C$ on a variable is added as a subscript to the $\OO$ notation. Moreover, $\phi_{\Sigma}$ denotes the centered $d$-variate normal density with covariance matrix $\Sigma$, defined, for all vectors $\bb{x} = (x_1, \ldots, x_d) \in \R^d$, by
\[
\phi_{\Sigma }(\bb{x}) = (2 \pi)^{-d/2} |\Sigma|^{-1/2} \exp ( - \bb{x}^{\top} \Sigma^{-1} \, \bb{x} /2),
\]
where $|\Sigma| = \varrho_1 \times \cdots \times \varrho_d / p_0$ is the determinant of $\Sigma$ and, for all integers $i, j \in \{1, \ldots, d \}$, $(\Sigma^{-1})_{ij} = \varrho_i^{-1} \ind_{\{i = j\}} - p_0$. Finally, for arbitrary $\bb{k} = (k_1, \ldots, k_d) \in \N_0^d$ and all $i \in \{1, \ldots, d \}$, introduce the notation
\[
\delta_{i, k_i} = (k_i - r \varrho_i) / \sqrt{r},
\]
and define the vector $\bb{\delta}_{\bb{k}} = (\delta_{1, k_1},\ldots,\delta_{d, k_d})$, whose $L_1$-norm is $\|\bb{\delta}_{\bb{k}}\|_1 = |\delta_{1, k_1}| + \cdots + |\delta_{d, k_d}|$.

\section{Main results\label{sec:2}}

Proposition~\ref{prop:1} provides an asymptotic expansion for the log-ratio of the negative multinomial probability mass function over the multivariate normal density function with the same mean-covariance structure, up to an error of order $r^{-3/2}$. The case $d = 1$ was handled recently in \cite[Lemma~1]{doi:10.1007/s00184-023-00897-2} in studying the asymptotics of the median of the negative binomial distribution to solve an open problem from \cite{MR4135709}. The result below refines substantially the pointwise convergence result due to \cite{MR0251828, MR1219279}, which did not include explicit error terms of any order.

\begin{proposition}
\label{prop:1}
For any real number $\gamma \in (0, \infty)$, let
\begin{equation}
\label{eq:2}
B_{r,\bb{p}}(\gamma) = \left\{\bb{k} \in \N_0^d : \max_{i\in \{1,\ldots,d\}} \bigg|\frac{\delta_{i,k_i}}{\sqrt{r} \varrho_i}\bigg| \le \gamma r^{-1/3} ~~\text{and}~~ \bigg|\sum_{i=1}^d \frac{\delta_{i,k_i}}{\sqrt{r} / p_0}\bigg| \le \gamma r^{-1/3}\right\}
\end{equation}
denote the bulk of the support of the negative multinomial distribution. Then, for all $\bb{k} \in B_{r,\bb{p}}(\gamma)$, one has, as $r \to \infty$,
\[
\ln \left \{ \frac{P_{r,\bb{p}}(\bb{k})}  {r^{-d/2} \phi_{\Sigma}(\bb{\delta}_{\bb{k}})} \right \} = r^{-1/2} F_{\bb{\varrho}, p_0, \bb{\delta}_{\bb{k}}} + r^{-1} S_{\bb{\varrho}, p_0, \bb{\delta}_{\bb{k}}} + \OO_{d,\bb{p},\gamma} \big \{ r^{-3/2} (1 + \|\bb{\delta}_{\bb{k}}\|_1^5) \big\},
\]
where
\[
F_{\bb{\varrho}, p_0, \bb{\delta}_{\bb{k}}} = - \frac{1}{2} \sum_{i=1}^d \delta_{i, k_i}  (\varrho_i^{-1} + p_0) + \frac{1}{6} \sum_{i, j, \ell = 1}^d \delta_{i,k_i} \delta_{j, k_j} \delta_{\ell, k_\ell}  (\varrho_i^{-2} \bb{1}_{\{i = j = \ell\}} - p_0^2)
\]
and
\begin{multline*}
S_{\bb{\varrho}, p_0, \bb{\delta}_{\bb{k}}} = \frac{1}{12} \left( p_0 - 1 - \sum_{i=1}^d \varrho_i^{-1}\right) + \frac{1}{4} \sum_{i, j=1}^d \delta_{i, k_i} \delta_{j, k_j}  (\varrho_i^{-2} \bb{1}_{\{i = j\}} + p_0^2)
\\ -  \frac{1}{12} \sum_{i, j, \ell, m=1}^d  \delta_{i,k_i} \delta_{j,k_j} \delta_{\ell,k_\ell} \delta_{m,k_m} ( \varrho_i^{-3} \bb{1}_{\{i = j = \ell = m\}} - p_0^3 ).
\end{multline*}
\end{proposition}

The refined local approximation of Proposition~\ref{prop:1} can be used to compute an upper bound on the Hellinger distance between the probability measures induced by the negative multinomial distribution, jittered by a $d$-dimensional uniform random vector on $(-1/2,1/2)^d$, and the corresponding centered multivariate normal distribution.
\begin{proposition}
\label{prop:2}
Let $\bb{K}\sim \mathcal{N\hspace{-0.8mm}M} (r,\bb{p})$ and $\bb{U}\sim \mathcal{U}\hspace{0.2mm}(-1/2,1/2)^d$ be independent random vectors. Define $\bb{X} = \bb{K} + \bb{U}$ and let $\widetilde{\PP}_{r,\bb{p}}$ be the law of $\bb{X}$. In particular, if $\PP_{r,\bb{p}}$ denotes the law of $\bb{K}$, then for any Borel set $B\in \mathscr{B}(\R^d)$, one has
\[
\widetilde{\PP}_{r,\bb{p}}(B) = \int_{\N_0^d} \int_{(-{1}/{2}, {1}/{2})^d} \ind_B(\bb{k} + \bb{u}) \rd \bb{u} \, \PP_{r,\bb{p}}(\rd \bb{k}).
\]
Moreover if $\QQ_{r,\bb{p}}$ denotes the measure on $\R^d$ induced by the multivariate normal law $\mathcal{N}_d(r \bb{\varrho}, r \Sigma)$, then for any vector $\bb{p}\in (0,1)^d$ that satisfies $\|\bb{p}\|_1 < 1$, there exists a universal positive constant $C\in (0,\infty)$ such that
\[
\mathcal{H}(\widetilde{\PP}_{r,\bb{p}}, \QQ_{r,\bb{p}}) \le \frac{C}{\sqrt{r p_0}} \sqrt{d^{\hspace{0.2mm}2} + \sum_{i=1}^d \frac{1}{p_i}} \le \frac{C \, d}{\sqrt{r}} \times \frac{2}{\min(p_0,\ldots, p_d)},
\]
where $\mathcal{H}$ denotes the Hellinger distance.
\end{proposition}

\begin{remark}
Given the relation between the Hellinger distance and other probability metrics, the bound given in Proposition~\ref{prop:2} remains valid for any and all of the discrepancy, Kolmogorov, L\'evy, Prokhorov, and total variation metrics.
\end{remark}

Next, for any scalar $b \in (0, \infty)$, consider the set $\Theta_b =  \{\bb{p}\in (0,1)^d: \|\bb{p}\|_1 < 1 ~~ \text{and} ~~ \min(p_0, \ldots, p_d) \geq b \}$, and define the statistical experiments $\mathscr{P} = \{\PP_{r,\bb{p}} : {\bb{p}\in \Theta_b}\}$ and $\mathscr{Q} = \{\QQ_{r,\bb{p}} : \bb{p}\in \Theta_b \}$, where $\PP_{r,\bb{p}}$ denotes the measure on $\N_0^d$ induced by the $\mathcal{N\hspace{-0.8mm}M} (r,\bb{p})$ distribution, and $\QQ_{r,\bb{p}}$ denotes the measure on $\R^d$ induced by the $\mathcal{N}_d(r \bb{\varrho}, r \Sigma)$ distribution.

The Le\,Cam distance between these two experiments, called $\Delta$-distance in \cite{MR1784901}, is then given by
\[
\Delta(\mathscr{P},\mathscr{Q}) = \max \{ \delta(\mathscr{P},\mathscr{Q}),\delta(\mathscr{Q},\mathscr{P}) \},
\]
where
\begin{equation}
\label{eq:3}
\begin{aligned}
\delta(\mathscr{P},\mathscr{Q})
= \inf_{T_1} \sup_{\bb{p}\in \Theta_b} \bigg\|\int_{\N_0^d} T_1(\bb{k}, \cdot \, ) \, \PP_{r,\bb{p}}(\rd \bb{k}) - \QQ_{r,\bb{p}}\bigg\|, \quad \delta(\mathscr{Q},\mathscr{P})
= \inf_{T_2} \sup_{\bb{p}\in \Theta_b} \bigg\|\PP_{r,\bb{p}} - \int_{\R^d} T_2(\bb{x}, \cdot \, ) \, \QQ_{r,\bb{p}}(\rd \bb{x})\bigg\|.
\end{aligned}
\end{equation}
Here, $\|\cdot\|$ denotes the total variation norm, and the infima are taken, respectively, over all Markov kernels $T_1 : \N_0^d \times \mathscr{B}(\R^d) \to [0, 1]$ and $T_2 : \R^d \times \mathscr{B}(\N_0^d) \to [0, 1]$. See Section~\ref{sec:3} for a brief motivation of the Le\,Cam distance, and refer to \cite{MR3850766} for a review of Le\,Cam's theory for the comparison of statistical models.

The following result provides an upper bound on the Le\,Cam distance between negative multinomial and multivariate normal experiments.
\begin{proposition}
\label{prop:3}
There exists a positive constant $C_b\in (0,\infty)$ that depends only on $b \in (0, \infty)$ such that $\Delta (\mathscr{P},\mathscr{Q})  \le C_b \, d / \sqrt{r}$.
\end{proposition}

Finally, consider the following multivariate normal experiments with independent components $\widetilde{\mathscr{Q}} = \{\widetilde{\QQ}_{r, \bb{p}} : \bb{p} \in \Theta_b \}$ and $\mathscr{Q}^{\star} = \{\QQ_{r, \bb{p}}^{\star} : \bb{p} \in \Theta_b \}$, where $\widetilde{\QQ}_{r, \bb{p}}$ denotes the measure on $\R^d$ induced by the $\mathcal{N}_d [r \bb{\varrho}, r \mathrm{diag}(\bb{\varrho})]$ distribution, and $\QQ_{r, \bb{p}}^{\star}$ denotes the measure on $\R^d$ induced by the $\mathcal{N}_d [ \sqrt{r \bb{\varrho}}, \mathrm{diag}(1/4,\ldots,1/4) ]$ distribution.

A straightforward adaptation of the argument in \cite[Section 7]{MR1922539} shows that there exists a positive constant $C_b\in (0,\infty)$ that depends only on $b \in (0, \infty)$ such that $\Delta(\mathscr{Q},\widetilde{\mathscr{Q}}) \le C_b \, \sqrt{{d}/{r}}$ and $\Delta(\widetilde{\mathscr{Q}},\mathscr{Q}^{\star}) \le C_b \, {d}/{\sqrt{r}}$, using a variance stabilizing transformation, with proper adjustments to the deficiencies in~\eqref{eq:3}.

The following corollary is a direct consequence of the last two bounds, together with Proposition~\ref{prop:3} and the triangle inequality for the pseudo-metric $\Delta$.

\begin{corollary}
\label{cor:1}
With the same notation as in Proposition~\ref{prop:3}, there exists a positive constant $C_b \in (0,\infty)$ that depends only on $b \in (0, \infty)$ such that $\Delta(\mathscr{P},\widetilde{\mathscr{Q}}) \le C_b \,  {d}/{\sqrt{r}}$ and $\Delta(\mathscr{P},\mathscr{Q}^{\star}) \le C_b \, {d}/{\sqrt{r}}$.
\end{corollary}

\section{Discussion\label{sec:3}}

Given that error terms are explicit up to order $1/r$, Proposition~\ref{prop:1} is a significant refinement of earlier results \cite{MR0251828, MR1219279}, where there were no explicit error terms of any order.

To fully appreciate the significance of Proposition~\ref{prop:3}, it is important to understand the role of the Le\,Cam distance in the comparison of two statistical models. In particular, what renders this concept useful is the realization that two seemingly dissimilar statistical experiments can lead to asymptotically equivalent inferences. This equivalence is achieved through the utilization of Markov kernels to transfer information from one setting to the other.

A key example of this, due to~\citet{MR1425959}, is the asymptotic equivalence between the density estimation problem and the Gaussian white noise problem. As the number of observations approaches infinity, the Le\,Cam distance between these two experiments tends to zero. At the heart of this discovery lies the idea that information obtained by sampling observations from an unknown density function, and counting the observations that fall in the various boxes of a fine partition of the density's support, can be encoded using the increments of a properly scaled Brownian motion with a drift. A simpler proof of this asymptotic equivalence over a larger class of densities was given by \citet{MR2102503}, who integrated a Haar wavelet cascade scheme with coupling inequalities that link the binomial and univariate normal distributions at each stage. For a brief review of Le\,Cam's theory for the comparison of statistical models, refer to \cite{MR3850766}.

\vspace{-2mm}

\appendix

\setcounter{section}{0}
\renewcommand{\thelemma}{\Alph{lemma}}

\section{Appendix A: Proofs\label{App:1}}

\begin{proof}[Proof of Proposition~\ref{prop:1}.]
Fix $\gamma \in (0, \infty)$, define the set $B_{r,\bb{p}}(\gamma)$ as in \eqref{eq:2}, and take logarithms in~\eqref{eq:1}. One then has, for every $\bb{k} = (k_1, \ldots, k_d) \in B_{r,\bb{p}}(\gamma)$,
\[
\ln \{ P_{r,\bb{p}}(\bb{k}) \} = \ln \{\Gamma(r + \|\bb{k}\|_1)\} - \ln \{\Gamma (r) \} - \sum_{i=1}^d \ln (k_i!) + (r + \|\bb{k}\|_1) \ln (p_0) + \sum_{i=1}^d k_i \ln (\varrho_i).
\]
First call on Stirling's formula \cite[p.\,257]{MR0167642}, which implies that, for all $x \in (0, \infty)$ and $k \in \mathbb{N}$, one has
\begin{align*}
\ln \{ \Gamma (x) \}
&= \frac{1}{2} \ln(2\pi) + (x - \tfrac{1}{2}) \ln (x) - x + \frac{1}{12x} + \OO(x^{-3}), \\
\ln (k!)
&= \frac{1}{2} \ln(2\pi) + (k + \tfrac{1}{2}) \ln (k) - k + \frac{1}{12k} + \OO(k^{-3}).
\end{align*}
Upon substitution, and after simple algebraic manipulations, one finds that, for all $\bb{k} \in B_{r,\bb{p}}(\gamma)$,
\begin{align*}
\ln \{ P_{r,\bb{p}} (\bb{k}) \}
&= -\frac{d}{2} \ln(2\pi) + (r + \|\bb{k}\|_1) \ln (r + \|\bb{k}\|_1) - r \ln (r) - \sum_{i=1}^d k_i \ln (k_i) \notag \\
& \quad - \frac{1}{2} \ln (r + \|\bb{k}\|_1) + \frac{1}{2} \ln (r) - \frac{1}{2} \sum_{i=1}^d \ln (k_i) + (r + \|\bb{k}\|_1) \ln (p_0) + \sum_{i=1}^d k_i \ln (\varrho_i) \notag \\
&\quad \quad + \frac{1}{12 r} \left(\frac{r}{r + \|\bb{k}\|_1} - 1 - \sum_{i=1}^d \frac{r}{k_i}\right) + \OO\left[\frac{1}{r^3} \left\{\frac{r^3}{(r + \|\bb{k}\|_1)^3} + 1 + \sum_{i=1}^d \frac{r^3}{k_i^3}\right\}\right].
\end{align*}

Given that, for each $i \in \{ 1, \ldots, d \}$, one has
\[
k_i = r \varrho_i  \left( 1 + \frac{\delta_{i,k_i}}{\sqrt{r} \varrho_i} \right), \quad r + \|\bb{k}\|_1  = \frac{r}{p_0} \left(1 + \sum_{i=1}^d \frac{\delta_{i,k_i}}{\sqrt{r} / p_0 }\right),
\]
further simplifications yield
\begin{equation}
\label{eq:4}
\begin{aligned}
\ln \{ P_{r,\bb{p}} (\bb{k}) \}
&= -\frac{1}{2} \ln \left \{ \frac{(2\pi r)^d}{p_0} \prod_{i=1}^d \varrho_i \right \} + A_r - B_r - C_r - D_r \\
& \quad+ \frac{1}{12 r} \left\{ p_0 \left(1 + \sum_{i=1}^d \frac{\delta_{i, k_i}}{\sqrt{r} / p_0}\right)^{-1} - 1 - \sum_{i=1}^d \varrho_i^{-1} \left(1 + \frac{\delta_{i, k_i}}{\sqrt{r} \varrho_i}\right)^{-1}\right\} + \OO_{d,\bb{p},\gamma}(r^{-3}),
\end{aligned}
\end{equation}
where
\[
A_r = \frac{r}{p_0} \left(1 + \sum_{i=1}^d \frac{\delta_{i,k_i}}{\sqrt{r} / p_0}\right) \ln \left(1 + \sum_{i=1}^d \frac{\delta_{i, k_i}}{\sqrt{r} / p_0}\right), \quad B_r = r \sum_{i=1}^d \varrho_i \left(1 + \frac{\delta_{i,k_i}}{\sqrt{r} \varrho_i}\right) \ln \left(1 + \frac{\delta_{i ,k_i}}{\sqrt{r} \varrho_i} \right),
\]
\[
C_r = \frac{1}{2} \ln \left(1 + \sum_{i=1}^d \frac{\delta_{i, k_i}}{\sqrt{r} / p_0}\right), \quad D_r = \frac{1}{2} \sum_{i=1}^d \ln \left(1 + \frac{\delta_{i, k_i}}{\sqrt{r} \varrho_i}\right).
\]

Next, call on the following Taylor expansions, valid for a suitable $\eta \in (0, 1)$ and all $x \in \mathbb{R}$ with $|x| < \eta < 1$:
\begin{equation}
\label{eq:5}
(1 + x) \ln(1 + x) = x + \frac{x^2}{2} - \frac{x^3}{6} + \frac{x^4}{12} + \OO \big \{(1 - \eta)^{-4} x^5 \big \},
\end{equation}
\begin{equation}
\label{eq:6}
\frac{1}{2} \ln(1 + x) = \frac{x}{2} - \frac{x^2}{4} + \OO \big \{(1 - \eta)^{-3} x^3 \big \}.
\end{equation}

Applying \eqref{eq:5} to $A_r$ and $B_r$, one finds, after simplification,
\begin{multline*}
A_r = \sum_{i=1}^d (k_i - r \varrho_i) + \frac{p_0}{2}\sum_{i, j=1}^d \delta_{i, k_i} \delta_{j, k_j} - \frac{p_0^2}{6\sqrt{r}} \sum_{i, j, \ell=1}^d \delta_{i, k_i} \delta_{j,k_j} \delta_{\ell, k_\ell} + \frac{p_0^3}{12r} \sum_{i, j, \ell, m=1}^d  \delta_{i, k_i} \delta_{j, k_j} \delta_{\ell, k_\ell} \delta_{m, k_m} + \OO_{d, \bb{p}, \gamma} \big( r^{-3/2} \|\bb{\delta}_{\bb{k}}\|_1^5 \big)
\end{multline*}
and
\begin{align*}
B_r = \sum_{i=1}^d (k_i - r \varrho_i) + \frac{1}{2} \sum_{i=1}^d \varrho_i^{-1} \delta_{i,k_i}^2  - \frac{1}{6\sqrt{r}} \sum_{i=1}^d \varrho_i^{-2} \delta_{i, k_i}^3 + \frac{1}{12r} \sum_{i=1}^d \varrho_i^{-3} \delta_{i,k_i}^4 + \OO_{d, \bb{p}, \gamma}\left( r^{-3/2} \|\bb{\delta}_{\bb{k}}\|_1^5 \right).
\end{align*}

Similarly, applying \eqref{eq:6} to $C_r$ and $D_r$, one finds, after simplification,
\[
C_r = \frac{p_0}{2 \sqrt{r}} \sum_{i=1}^d \delta_{i,k_i} - \frac{p_0^2}{4 r} \sum_{i,j=1}^d \delta_{i,k_i} \delta_{j,k_j} + \OO_{d,\bb{p},\gamma}\ \big( r^{-3/2} \|\bb{\delta}_{\bb{k}} \|_1^3 \big)
\]
and
\[
D_r =  \frac{1}{2 \sqrt{r}} \sum_{i=1}^d \varrho_i^{-1} \delta_{i,k_i} - \frac{1}{4 r} \sum_{i=1}^d \varrho_i^{-2} \delta_{i,k_i}^2 + \OO_{d,\bb{p},\gamma} \big( r^{-3/2} \|\bb{\delta}_{\bb{k}}\|_1^3 \big).
\]

Substituting these four expressions into~\eqref{eq:4}, and exploiting the explicit form for $\Sigma^{-1}$ given in Section~\ref{sec:1} for the joint density $\phi_{\Sigma}(\bb{\delta}_{\bb{k}})$, one can express the difference $\ln \{ P_{r,\bb{p}} (\bb{k}) \} - \ln \{r^{-d/2} \phi_{\Sigma}(\bb{\delta}_{\bb{k}}) \}$, for every $\bb{k} = (k_1, \ldots, k_d) \in B_{r,\bb{p}}(\gamma)$, as
\begin{multline*}
\frac{1}{6 \sqrt{r}} \sum_{i, j, \ell=1}^d \delta_{i,k_i} \delta_{j,k_j} \delta_{\ell,k_\ell} \big( \varrho_i^{-2} \bb{1}_{\{i = j = \ell\}} - p_0^2 \big) - \frac{1}{12 r} \sum_{i, j, \ell, m=1}^d  \delta_{i,k_i} \delta_{j,k_j} \delta_{\ell,k_\ell} \delta_{m,k_m} \big( \varrho_i^{-3} \bb{1}_{\{i = j = \ell = m\}} - p_0^3 \big) \\
- \frac{1}{2 \sqrt{r}}\sum_{i=1}^d \delta_{i, k_i} (\varrho_i^{-1} + p_0) + \frac{1}{4r} \sum_{i,j=1}^d \delta_{i,k_i} \delta_{j,k_j}  (\varrho_i^{-2} \bb{1}_{\{i = j\}} + p_0^2) + \frac{1}{12 r} \left(p_0 - 1 - \sum_{i=1}^d \varrho_i^{-1}\right) + \OO_{d,\bb{p},\gamma} \big\{ r^{-3/2} (1 + \|\bb{\delta}_{\bb{k}}\|_1^5) \big\}.
\end{multline*}
After grouping the terms with respect to the powers of $r$, one can conclude.\hfill $\Box$
\end{proof}

\begin{proof}[Proof of Proposition~\ref{prop:2}.]
Recall the definition of the bulk $B_{r,\bb{p}} (1)$ from \eqref{eq:2} and define the set
\[
\widetilde{B}_{r,\bb{p}}(1) = \big \{\bb{x}\in \R^d : \exists_{(\bb{k},\bb{u})\in B_{r,\bb{p}}(1) \times (-1/2,1/2)^d} \; \bb{x} = \bb{k} + \bb{u} \big \}.
\]
By the bound on the Hellinger distance found at the bottom of p.~726 of \cite{MR1922539}, it is already known that
\begin{equation}
\label{eq:7}
\mathcal{H} \Big(\widetilde{\PP}_{r,\bb{p}},\QQ_{r,\bb{p}} \Big) \le \sqrt{2 \, \PP \big \{\bb{X}\in \widetilde{B}_{r,\bb{p}}^c(1) \big \} + \EE\left[\ln\left\{\frac{\rd \widetilde{\PP}_{r,\bb{p}}}{\rd \QQ_{r,\bb{p}}}(\bb{X})\right\} \, \ind_{\widetilde{B}_{r,\bb{p}}(1)}(\bb{X})\right]}.
\end{equation}

For each $i \in \{ 1, \ldots, d \}$, the marginal distribution of $K_i$ is $\mathcal{N\hspace{-0.8mm}M} [r, p_i / (p_i + p_0)]$, so that $\EE(K_i) = r \varrho_i$ and $\Var(K_i) = r \varrho_i (1 + \varrho_i)$. By applying a union bound, one finds
\begin{align*}
\PP \Big \{\bb{X}\in \widetilde{B}_{r,\bb{p}}^c(1) \Big \}
&\le \sum_{i=1}^d \PP \big (|K_i - r \varrho_i| > \varrho_i r^{2/3} \big) + \PP\left(\sum_{i=1}^d |K_i - r \varrho_i| > p_0^{-1} r^{2/3}\right) \\
&\le \sum_{i=1}^d \PP\left(\frac{|K_i - r \varrho_i|}{\sqrt{r \varrho_i (1 + \varrho_i)}} > \frac{\varrho_i r^{2/3}}{\sqrt{r \varrho_i (1 + \varrho_i)}}\right) + \sum_{i=1}^d \PP\left(\frac{|K_i - r \varrho_i|}{\sqrt{r \varrho_i (1 + \varrho_i)}} > \frac{p_0^{-1} r^{2/3}}{d \sqrt{r \varrho_i (1 + \varrho_i)}}\right) .
\end{align*}
Applying a concentration inequality to each of the right-hand summands, one deduces that, for sufficiently large $r$,
\begin{equation}
\label{eq:8}
\begin{aligned}
\PP \Big \{\bb{X}\in \widetilde{B}_{r,\bb{p}}^c(1) \Big \}
&\le 2 d \, \max_{i\in \{1,\ldots,d\}} \PP\left(\frac{|K_i - r \varrho_i|}{\sqrt{r \varrho_i (1 + \varrho_i)}} > \frac{1}{d} \sqrt{\min(p_i, p_i^{-1})} \, r^{1/6}\right) \\
&\le 100 \, d \exp\left\{-\frac{\min(p_1,\ldots,p_d)}{100 \, d^{\hspace{0.2mm}2} \max(p_1,\ldots,p_d)} \, r^{1/3}\right\}.
\end{aligned}
\end{equation}

For the expectation in~\eqref{eq:7}, if $\widetilde{P}_{r,\bb{p}}(\bb{x})$ denotes the density function associated with $\widetilde{\PP}_{r,\bb{p}}$ at any $\bb{x} \in \mathbb{R}^d$, i.e., it is equal to $P_{r,\bb{p}}(\bb{k})$ whenever $\bb{k}\in \N_0^d$ is closest to $\bb{x}$, then one can express
\begin{equation}
\label{eq:9}
\EE\left[\ln\left\{\frac{\rd \widetilde{\PP}_{r,\bb{p}}}{\rd \QQ_{r,\bb{p}}}(\bb{X})\right\} \, \ind_{\widetilde{B}_{r,\bb{p}}(1)}(\bb{X})\right]
= \EE\left[\ln\left\{\frac{\widetilde{P}_{r,\bb{p}}(\bb{X})}{r^{-d/2} \phi_{\Sigma}(\bb{\delta}_{\bb{X}})}\right\} \, \ind_{\widetilde{B}_{r,\bb{p}}(1)}(\bb{X})\right]
\equiv (\mathrm{I}) + (\mathrm{II}) + (\mathrm{III})
\end{equation}
as a sum of three terms, namely
\begin{align*}
\mathrm{(I)} = \EE\left[\ln\left\{\frac{P_{r,\bb{p}}(\bb{K})}{r^{-d/2} \phi_{\Sigma}(\bb{\delta}_{\bb{K}})}\right\} \, \ind_{B_{r,\bb{p}}(1)}(\bb{K})\right] , \quad
\mathrm{(II)} = \EE\left[\ln\left\{\frac{r^{-d/2} \phi_{\Sigma}(\bb{\delta}_{\bb{K}})}{r^{-d/2} \phi_{\Sigma}(\bb{\delta}_{\bb{X}})}\right\} \, \ind_{B_{r,\bb{p}}(1)}(\bb{K}) \right],
\end{align*}
and
\[
\mathrm{(III)} =  \EE\left[\ln\left\{\frac{P_{r,\bb{p}}(\bb{K})}{r^{-d/2} \phi_{\Sigma}(\bb{\delta}_{\bb{X}})}\right\} \, \big \{\ind_{\widetilde{B}_{r,\bb{p}}(1)}(\bb{X}) - \ind_{B_{r,\bb{p}}(1)}(\bb{K}) \big \}\right] .
\]
By Proposition~\ref{prop:1}, one has
\begin{align}
\label{eq:10}
(\mathrm{I}) & = - r^{-1/2} \sum_{i=1}^d \frac{1}{2} \, \EE \big \{\delta_{i,K_i} \ind_{B_{r,\bb{p}}(1)}(\bb{K}) \big\} (\varrho_i^{-1} - p_0 ) + r^{-1/2} \sum_{i,j,\ell=1}^d \frac{1}{6} \, \EE \big \{\delta_{i,K_i} \delta_{j,K_j} \delta_{\ell,K_\ell} \ind_{B_{r,\bb{p}}(1)}(\bb{K}) \big \} (\varrho_i^{-2} \bb{1}_{\{i = j = \ell\}} - p_0^2 ) \notag \\
 & \quad + r^{-1} \, \OO ( T_{\bb{\varrho}, p_0}) + \OO_{d,\bb{p}} (r^{-3/2}),
\end{align}
where
\begin{multline*}
T_{\bb{\varrho}, p_0} = 1 + \sum_{i=1}^d \varrho_i^{-1} + \sum_{i,j=1}^d \{\EE(\delta_{i,K_i}^2)\}^{1/2} \{\EE(\delta_{j,K_j}^2)\}^{1/2} (\varrho_i^{-2} \bb{1}_{\{i = j\}} + p_0^2) \Bigg. \\ + \sum_{i,j,\ell,m=1}^d \{\EE(\delta_{i,K_i}^4)\}^{1/4} \{\EE(\delta_{j,K_j}^4)\}^{1/4} \{\EE(\delta_{\ell,K_{\ell}}^4)\}^{1/4} \{\EE(\delta_{m,K_m}^4)\}^{1/4} \Bigg. \big (\varrho_i^{-3} \bb{1}_{\{i = j = \ell = m\}} + p_0^3 \big).
\end{multline*}

The term $\OO ( T_{\bb{\varrho}, p_0})$ is crucial to get the correct upper bound on the Le\,Cam distance in Proposition~\ref{prop:3}. By Lemma~\ref{lem:A} in Appendix~\ref{App:2}, one has
\[
\sum_{i,j=1}^d \{\EE(\delta_{i,K_i}^2)\}^{1/2} \{\EE(\delta_{j,K_j}^2)\}^{1/2} (\varrho_i^{-2} \bb{1}_{\{i = j\}} + p_0^2) \le \sum_{i,j=1}^d (\varrho_i p_0^{-1})^{1/2} (\varrho_j p_0^{-1})^{1/2} (\varrho_i^{-2} \bb{1}_{\{i = j\}} + p_0^2),
\]
and the right-hand term can be rewritten and further bounded above as follows:
\[
\sum_{i=1}^d p_i^{-1} + \sum_{i,j=1}^d (p_i p_j)^{1/2} = \sum_{i=1}^d p_i^{-1} + \left(\sum_{i=1}^d p_i^{1/2}\right)^2
\le \sum_{i=1}^d p_i^{-1} + d \sum_{i=1}^d p_i \le d + \sum_{i=1}^d p_i^{-1}.
\]
Similarly, by Lemma~\ref{lem:A} in Appendix~\ref{App:2}, one has
\begin{multline*}
\sum_{i,j,\ell,m=1}^d \{\EE(\delta_{i,K_i}^4)\}^{1/4} \{\EE(\delta_{j,K_j}^4)\}^{1/4} \{\EE(\delta_{\ell,K_{\ell}}^4)\}^{1/4} \{\EE(\delta_{m,K_m}^4)\}^{1/4} (\varrho_i^{-3} \bb{1}_{\{i = j = \ell = m\}} + p_0^3) \\
\le \sum_{i,j,\ell,m=1}^d (4 \varrho_i^2 p_0^{-2})^{1/4} (4 \varrho_j^2 p_0^{-2})^{1/4} (4 \varrho_{\ell}^2 p_0^{-2})^{1/4} (4 \varrho_m^2 p_0^{-2})^{1/4} (\varrho_i^{-3} \bb{1}_{\{i = j = \ell = m\}} + p_0^3),
\end{multline*}
and the right-hand term can be rewritten and further bounded above as follows:
\begin{align*}
4 p_0^{-1} \left\{\sum_{i=1}^d p_i^{-1} + \sum_{i,j,\ell,m=1}^d (p_i p_j p_{\ell} p_m)^{1/2}\right\} & = 4 p_0^{-1} \left\{\sum_{i=1}^d p_i^{-1} + \left(\sum_{i=1}^d p_i^{1/2}\right)^4\right\} \\
& \le 4 p_0^{-1} \left\{\sum_{i=1}^d p_i^{-1} + d^{\hspace{0.2mm}2} \left(\sum_{i=1}^d p_i\right)^2\right\}
\le 4 p_0^{-1} \left(d^{\hspace{0.2mm}2} + \sum_{i=1}^d p_i^{-1}\right).
\end{align*}
Therefore, one can rewrite~\eqref{eq:10} as
\begin{multline}
\label{eq:11}
(\mathrm{I}) = - r^{-1/2} \, \sum_{i=1}^d \frac{1}{2} \, \EE \big \{ \delta_{i,K_i} \ind_{B_{r,\bb{p}}(1)}(\bb{K}) \big \} \big( \varrho_i^{-1} - p_0 \big) \\
+ r^{-1/2} \sum_{i, j, \ell=1}^d \frac{1}{6} \, \EE \big \{\delta_{i,K_i} \delta_{j,K_j} \delta_{\ell,K_\ell} \ind_{B_{r,\bb{p}}(1)}(\bb{K}) \big\} \big(\varrho_i^{-2} \bb{1}_{\{i = j = \ell\}} - p_0^2 \big)
 + (r p_0)^{-1} \, \OO \left(d^{\hspace{0.2mm}2} + \sum_{i=1}^d p_i^{-1}\right).
\end{multline}
By estimating the two expectations using Lemma~\ref{lem:B} in Appendix~\ref{App:2}, one finds
\begin{equation}
\label{eq:12}
\begin{aligned}
(\mathrm{I})
&= r^{-1/2} \, \OO_{d,\bb{p}} \left( p_0^{-1} \left[\PP \big\{ \bb{K}\in B_{r,\bb{p}}^c(1) \big \}\right]^{1/2}\right) \\[1mm]
&\quad+ r^{-1/2} \sum_{i,j,\ell=1}^d \frac{1}{6} \, r^{-1/2}
\begin{pmatrix}
2 \varrho_i \varrho_j \varrho_{\ell} + \ind_{\{i = j\}} \varrho_i \varrho_{\ell} + \ind_{\{j = \ell\}} \varrho_i \varrho_j \\
+ \ind_{\{i = \ell\}} \varrho_j \varrho_{\ell} + \ind_{\{i = j = \ell\}} \varrho_i
\end{pmatrix}
(\varrho_i^{-2} \bb{1}_{\{i = j = \ell\}} - p_0^2 ) \\
&\quad \quad + r^{-1/2} \, \OO_{d,\bb{p}} \left(p_0^{-3} \left[\PP \big\{ \bb{K}\in B_{r,\bb{p}}^c(1) \big \} \right]^{1/4} \right) + (r p_0)^{-1} \, \OO \left( d^{\hspace{0.2mm}2} + \sum_{i=1}^d p_i^{-1}\right) \\
&= r^{-1/2} \, \OO_{d,\bb{p}} \left(p_0^{-3} \left[\PP \big \{ \bb{K}\in B_{r,\bb{p}}^c(1) \big\} \right]^{1/4}\right) + (r p_0)^{-1} \,  \OO \left(d^{\hspace{0.2mm}2} + \sum_{i=1}^d p_i^{-1}\right).
\end{aligned}
\end{equation}

For the term $(\mathrm{II})$, note that
\begin{align*}
\ln\left\{\frac{r^{-d/2} \phi_{\Sigma}(\bb{\delta}_{\bb{K}})}{r^{-d/2} \phi_{\Sigma}(\bb{\delta}_{\bb{X}})}\right\}
&= r^{-1} (\bb{X} - r \bb{\varrho})^{\top} \Sigma^{-1} (\bb{X} - r \bb{\varrho})/2 - r^{-1} (\bb{K} - r \bb{\varrho})^{\top} \Sigma^{-1} (\bb{K} - r \bb{\varrho})/2 \\[-1mm]
&= r^{-1} (\bb{X} - \bb{K})^{\top} \Sigma^{-1} (\bb{X} - \bb{K})/2 + r^{-1} (\bb{X} - \bb{K})^{\top} \Sigma^{-1} (\bb{K} - r \bb{\varrho}).
\end{align*}
With the assumption that $\bb{K}$ and $\add{\bb{U} =} \bb{X} - \bb{K}\sim \mathcal{U} (-1/2,1/2)^d$ are independent, and using Lemma~\ref{lem:B} and the explicit form for $\Sigma^{-1}$ given in Section~\ref{sec:1}, one finds
\begin{align}
\label{eq:13}
|(\mathrm{II})|
&= \left|\EE\left[ \big \{ r^{-1} (\bb{X} - \bb{K})^{\top} \Sigma^{-1} (\bb{X} - \bb{K})/2 + r^{-1} (\bb{X} - \bb{K})^{\top} \Sigma^{-1} (\bb{K} - r \bb{\varrho}) \big \} \ind_{B_{r,\bb{p}}(1)}(\bb{K})\right]\right| \notag \\
&\le r^{-1} \frac{1}{2} \, \EE \big \{|(\bb{X} - \bb{K})^{\top} \Sigma^{-1} (\bb{X} - \bb{K})| \big\} + r^{-1/2} \, \EE \big \{|(\bb{X} - \bb{K})^{\top} \Sigma^{-1} \bb{\delta}_{\bb{K}}| \ind_{B_{r,\bb{p}}(1)}(\bb{K}) \big \} \notag \\[-1mm]
&\le r^{-1} \sum_{i,j=1}^d \frac{1}{8} \big| \varrho_i^{-1} \ind_{\{i=j\}} - p_0 \big| + r^{-1/2} \sum_{i,j=1}^d \frac{1}{2} \big| \varrho_i^{-1} \ind_{\{i=j\}} - p_0 \big| \, \EE\big\{|\delta_{j,K_j}| \ind_{B_{r,\bb{p}}(1)}(\bb{K})\big\} \notag \\[-1mm]
&\le r^{-1} \sum_{i,j=1}^d \frac{1}{8} \big| \varrho_i^{-1} \ind_{\{i=j\}} - p_0 \big| + r^{-1/2} \sum_{i,j=1}^d \frac{1}{2} \big| \varrho_i^{-1} \ind_{\{i=j\}} - p_0 \big| \sqrt{\EE\big\{\delta_{j,K_j}^2 \ind_{B_{r,\bb{p}}(1)}(\bb{K})\big\}} \notag \\[-1mm]
&= r^{-1} \, \OO\left(d^{\hspace{0.2mm}2} + \sum_{i=1}^d p_i^{-1}\right) + \OO_{d,\bb{p}}\left(r^{-1/2} \left[\PP \big \{\bb{K}\in B_{r,\bb{p}}^c(1) \big \}\right]^{1/2}\right).
\end{align}

For the term $(\mathrm{III})$, one can use the following very rough bound from Proposition~\ref{prop:1}, viz.
\[
\ln\left\{\frac{P_{r,\bb{p}}(\bb{K})}{r^{-d/2} \phi_{\Sigma}(\bb{\delta}_{\bb{K}})}\right\} \, \ind_{B_{r,\bb{p}}(1)}(\bb{K}) = \OO_{d,\bb{p}}(1).
\]
Then
\begin{equation}
\label{eq:14}
(\mathrm{III}) = \OO_{d,\bb{p}}\left[\EE\left\{|\ind_{\widetilde{B}_{r,\bb{p}}(1)}(\bb{X}) - \ind_{B_{r,\bb{p}}(1)}(\bb{K})|\right\}\right].
\end{equation}
Now, using the concentration inequality~\eqref{eq:8}, one has the exponential bound
\begin{align*}
\EE \Big \{|\ind_{\widetilde{B}_{r,\bb{p}}(1)}(\bb{X}) - \ind_{B_{r,\bb{p}}(1)}(\bb{K})| \Big \}
&= \EE \Big \{|\ind_{\widetilde{B}_{r,\bb{p}}(1) \times B_{r,\bb{p}}^c(1)}(\bb{X},\bb{K}) - \ind_{\widetilde{B}_{r,\bb{p}}^c(1) \times B_{r,\bb{p}}(1)}(\bb{X},\bb{K})| \Big \} \\[1.5mm]
&\le \PP \big \{\bb{K}\in B_{r,\bb{p}}^c(1) \big \} + \PP \big \{\bb{X}\in \widetilde{B}_{r,\bb{p}}^c(1) \big \} \\[1mm]
&= \OO_d\left[\exp\left\{-\frac{\min(p_1,\ldots,p_d)}{100 \, d^{\hspace{0.2mm}2} \max(p_1,\ldots,p_d)} \, r^{1/3}\right\}\right].
\end{align*}
Therefore, by putting the estimates \eqref{eq:12}--\eqref{eq:14} back into~\eqref{eq:9}, one has, as $r\to \infty$,
\begin{align*}
\EE\left[\ln\left\{\frac{\rd \widetilde{\PP}_{r,\bb{p}}}{\rd \QQ_{r,\bb{p}}}(\bb{X})\right\} \, \ind_{\widetilde{B}_{r,\bb{p}}(1)}(\bb{X})\right]
&= (\mathrm{I}) + (\mathrm{II}) + (\mathrm{III}) = (r p_0)^{-1} \, \OO \left( d^{\hspace{0.2mm}2} + \sum_{i=1}^d p_i^{-1} \right).
\end{align*}
This result, together with the concentration inequality~\eqref{eq:8} applied in~\eqref{eq:7}, leads to the desired conclusion.\hfill $\Box$
\end{proof}

\begin{proof}[Proof of Proposition~\ref{prop:3}.]
By Proposition~\ref{prop:2}, one obtains the desired upper bound on $\delta(\mathscr{P},\mathscr{Q})$ by choosing the Markov kernel $T_1^{\star}$ that adds the uniform jittering $\bb{U}$ to $\bb{K}$. The latter is defined, for arbitrary $\bb{k}\in \N_0^d$ and $B\in \mathscr{B}(\R^d)$, by
\[
T_1^{\star}(\bb{k},B) = \int_{(-1/2, 1/2)^d} \ind_B(\bb{k} + \bb{u}) \rd \bb{u} .
\]
To get the upper bound on $\delta(\mathscr{Q},\mathscr{P})$, it suffices to consider a Markov kernel $T_2^{\star}$ that inverts the effect of $T_1^{\star}$, i.e., rounding off every component of $\bb{X}\sim \mathcal{N}_d (r \bb{\varrho}, r \Sigma)$ to the nearest integer. Then, as explained in Section~5 of \cite{MR1922539}, one has
\begin{align*}
\delta(\mathscr{Q},\mathscr{P})
&\le \bigg\|\PP_{r,\bb{p}} - \int_{\R^d} T_2^{\star}(\bb{x}, \cdot \, ) \, \QQ_{r,\bb{p}}(\rd \bb{x})\bigg\| = \bigg\|\int_{\R^d} T_2^{\star}(\bb{x}, \cdot \, ) \int_{\N_0^d} T_1^{\star}(\bb{k}, \rd \bb{x}) \, \PP_{r,\bb{p}}(\rd \bb{k}) - \int_{\R^d} T_2^{\star}(\bb{x}, \cdot \, ) \, \QQ_{r,\bb{p}}(\rd \bb{x})\bigg\| \\
&\le \bigg\|\int_{\N_0^d} T_1^{\star}(\bb{k}, \cdot \, ) \, \PP_{r,\bb{p}}(\rd \bb{k}) - \QQ_{r,\bb{p}}\bigg\|,
\end{align*}
yielding the same upper bound as on $\delta(\mathscr{P},\mathscr{Q})$ by Proposition~\ref{prop:2}. This concludes the proof.\hfill $\Box$
\end{proof}

\section{Appendix B: Moment estimates\label{App:2}}

The first lemma below gives some central moments of order 1--4 for the negative multinomial distribution. These moment formulas can be proved using the general moment formulas recently derived in \cite{arXiv:2209.04733} or using a symbolic calculator like Mathematica. The lemma is used to estimate the terms of order $r^{-1}$ in~\eqref{eq:10} of the proof of Proposition~\ref{prop:2}, and also as a preliminary result for the proof of Lemma~\ref{lem:B} below.

\begin{lemma}\label{lem:A}
If $\bb{K} = (K_1,\ldots,K_d)\sim \mathcal{N\hspace{-0.8mm}M} (r,\bb{p})$ according to~\eqref{eq:1}, then, for all integers $i, j, \ell \in \{1, \ldots, d\}$, one has
\begin{align*}
&\EE(\delta_{i,K_i}) = 0, \quad \EE(\delta_{i,K_i} \delta_{j,K_j}) = \varrho_i \ind_{\{i = j\}} + \varrho_i \varrho_j, \\[1mm]
&\EE(\delta_{i,K_i} \delta_{j,K_j} \delta_{\ell,K_{\ell}}) = r^{-1/2} \left(2 \varrho_i \varrho_j \varrho_{\ell} + \ind_{\{i = j\}} \varrho_i \varrho_{\ell} + \ind_{\{j = \ell\}} \varrho_i \varrho_j + \ind_{\{i = \ell\}} \varrho_j \varrho_{\ell} + \ind_{\{i = j = \ell\}} \varrho_i\right), \\[1mm]
&\EE(\delta_{i,K_i}^4) = 3 \varrho_i^2 \, (1 + \varrho_i)^2 + \OO(r^{-1}). \notag
\end{align*}
In particular, for sufficiently large values of $r$, one has, for all $i \in \{ 1, \ldots, d\}$, $\EE (\delta_{i,K_i}^2 ) \le \varrho_i p_0^{-1}$ and $\EE (\delta_{i,K_i}^4) \le 4 \varrho_i^2 p_0^{-2}$.
\end{lemma}

One can also estimate the central moments of Lemma~\ref{lem:A} on various measurable events. The result below is used to estimate the terms of order $r^{-1/2}$ in~\eqref{eq:11} and one expectation in \eqref{eq:13}, both in the proof of Proposition~\ref{prop:2}.

\begin{lemma}\label{lem:B}
Let $B\in \mathscr{B}(\R^d)$ be a Borel set.
If $\bb{K} = (K_1,\ldots,K_d)\sim \mathcal{N\hspace{-0.8mm}M} (r,\bb{p})$ according to~\eqref{eq:1}, then, for all $i, j, \ell \in \{1, \ldots, d \}$, one has, for sufficiently large values of $r$,
\begin{align*}
&\left|\EE \big \{\delta_{i,K_i} \, \ind_B(\bb{K}) \big\} \right| \le p_0^{-1} \left\{\PP(\bb{K}\in B^c )\right\}^{1/2}, \\[1mm]
&\left|\EE \big\{ \delta_{i,K_i} \delta_{j,K_j} \, \ind_B(\bb{K}) \big\} - (\varrho_i \ind_{\{i = j\}} + \varrho_i \varrho_j)\right| \le 2 p_0^{-2} \left\{\PP(\bb{K}\in B^c )\right\}^{1/2}, \\[1mm]
&\left|\EE \big \{\delta_{i,K_i} \delta_{j,K_j} \delta_{\ell,K_{\ell}} \, \ind_B(\bb{K}) \big\} - r^{-1/2}
\begin{pmatrix}
2 \varrho_i \varrho_j \varrho_{\ell} + \ind_{\{i = j\}} \varrho_i \varrho_{\ell} + \ind_{\{j = \ell\}} \varrho_i \varrho_j \\
+ \ind_{\{i = \ell\}} \varrho_j \varrho_{\ell} + \ind_{\{i = j = \ell\}} \varrho_i
\end{pmatrix}
\right| \le 4 p_0^{-3} \left\{\PP(\bb{K}\in B^c )\right\}^{1/4}.
\end{align*}
\end{lemma}

\begin{proof}[Proof of Lemma~\ref{lem:B}]
Using H\"older's inequality and Lemma~\ref{lem:A}, one has, for $i \in \{1, \ldots, d \}$,
\begin{align*}
\left|\EE \{\delta_{i,K_i} \, \ind_B(\bb{K}) \}\right| = \left|\EE \{\delta_{i,K_i} \, \ind_{B^c}(\bb{K}) \}\right| \le \big \{\EE(\delta_{i,K_i}^2) \big\}^{1/2} \left\{\PP(\bb{K}\in B^c ) \right\}^{1/2} \le p_0^{-1} \left\{\PP(\bb{K}\in B^c ) \right\}^{1/2}.
\end{align*}
Similarly, one has, for all integers $i, j \in \{1, \ldots, d \}$,
\[
\left|\EE \big\{\delta_{i,K_i} \delta_{j,K_j} \, \ind_B(\bb{K}) \big \} - (\varrho_i \ind_{\{i = j\}} + \varrho_i \varrho_j) \right|
= \left|\EE \big \{\delta_{i,K_i} \delta_{j,K_j} \, \ind_{B^c }(\bb{K}) \big \}\right|,
\]
and, for sufficiently large values of $r$, this term is bounded above by
\[
\big \{\EE(\delta_{i,K_i}^4) \big \}^{1/4} \big \{\EE(\delta_{j,K_j}^4) \big \}^{1/4} \left\{\PP(\bb{K}\in B^c )\right\}^{1/2} \le (4 p_0^{-4})^{1/4} (4 p_0^{-4})^{1/4} \left\{\PP(\bb{K}\in B^c )\right\}^{1/2} \le 2 p_0^{-2} \left\{\PP(\bb{K}\in B^c )\right\}^{1/2}.
\]
Moreover, for all $i, j, \ell \in \{1, \ldots, d \}$,
\[
\left|\EE \big \{\delta_{i,K_i} \delta_{j,K_j} \delta_{\ell,K_{\ell}} \, \ind_B(\bb{K}) \big \} - r^{-1/2}
\begin{pmatrix}
2 \varrho_i \varrho_j \varrho_{\ell} + \ind_{\{i = j\}} \varrho_i \varrho_{\ell} + \ind_{\{j = \ell\}} \varrho_i \varrho_j \\
+ \ind_{\{i = \ell\}} \varrho_j \varrho_{\ell} + \ind_{\{i = j = \ell\}} \varrho_i
\end{pmatrix}
\right| = \left|\EE \big \{\delta_{i,K_i} \delta_{j,K_j} \delta_{\ell,K_{\ell}} \, \ind_{B^c }(\bb{K}) \big \}\right|,
\]
and, for sufficiently large values of $r$, this term is bounded above by
\begin{multline*}
\big \{ \EE(\delta_{i,K_i}^4) \big \}^{1/4} \big \{\EE(\delta_{j,K_j}^4) \big \}^{1/4} \big \{\EE(\delta_{\ell,K_{\ell}}^4) \big \}^{1/4} \big \{\PP(\bb{K}\in B^c ) \big \}^{1/4} \\
\le (4 p_0^{-4})^{1/4} (4 p_0^{-4})^{1/4} (4 p_0^{-4})^{1/4} \big \{\PP(\bb{K}\in B^c ) \big \}^{1/4} \le 4 p_0^{-3} \big \{\PP(\bb{K}\in B^c ) \big \}^{1/4}.
\end{multline*}
This concludes the proof.\hfill $\Box$
\end{proof}

\end{document}